\documentclass[12pt, reqno]{amsart}
\usepackage{amssymb, amstext, amscd, amsmath, amsthm}
\usepackage{color}
\usepackage{hyperref} 

\usepackage{amsthm}
\theoremstyle{plain}

\theoremstyle{plain}
\newtheorem{theorem}{Theorem}[section]
\newtheorem{thm}[theorem]{Theorem}
\newtheorem*{thm*}{Theorem}
\newtheorem{cor}[theorem]{Corollary}
\newtheorem{prop}[theorem]{Proposition}
\newtheorem{lem}[theorem]{Lemma}
\newtheorem{defn}[theorem]{Definition}

\theoremstyle{definition}
\newtheorem{rem}[theorem]{Remark}

\newcommand{\bC}{{\mathbb{C}}}

\newcommand{\bF}{{\mathbb{F}}}

\newcommand{\bN}{{\mathbb{N}}}

\newcommand{\bQ}{{\mathbb{Q}}}
\newcommand{\bR}{{\mathbb{R}}}

\newcommand{\bT}{{\mathbb{T}}}

\newcommand{\bZ}{{\mathbb{Z}}}

\newcommand{ \bm}{{\mathbf{m}}}
\newcommand{ \bn}{{\mathbf{n}}}
\newcommand{\bt}{{\mathbf{t}}}

  \newcommand{\A}{{\mathcal{A}}}

  \newcommand{\E}{{\mathcal{E}}}

\renewcommand{\O}{{\mathcal{O}}}
\renewcommand{\P}{{\mathcal{P}}}

\renewcommand{\S}{{\mathcal{S}}}
  
  \newcommand{\U}{{\mathcal{U}}}
  \newcommand{\V}{{\mathcal{V}}}
  
  \newcommand{\X}{{\mathcal{X}}}
  
  \newcommand{\Z}{{\mathcal{Z}}}

\newcommand{\fA}{{\mathfrak{A}}}

\newcommand{\fF}{{\mathfrak{F}}}

\newcommand{\fM}{{\mathfrak{M}}}

\newcommand{\fs}{{\mathfrak{s}}}
\newcommand{\ft}{{\mathfrak{t}}}
\newcommand{\fu}{{\mathfrak{u}}}

\newcommand{\Bm}{{\mathbf{m}}}

\newcommand{\rC}{{\mathrm{C}}}

\newcommand{\ep}{\varepsilon}
\renewcommand{\phi}{\varphi}
\newcommand{\upchi}{{\raise.35ex\hbox{\ensuremath{\chi}}}}

\newcommand{\qforal}{\quad\text{for all}\quad}

\newcommand{\Ad}{\operatorname{Ad}}

\newcommand{\id}{{\operatorname{id}}}

\newcommand{\ran}{\operatorname{Ran}}
\newcommand{\spn}{\operatorname{span}}

\newcommand{\rank}{\operatorname{rank}}
\newcommand{\cl}{\operatorname{cl}}

\newcommand{\ca}{\mathrm{C}^*}

\newcommand{\Fn}{\mathbb{F}_n^+}

\newcommand{\Fth}{\mathbb{F}_\theta^+}

\newcommand{\mt}{\varnothing}

\newcommand{\ol}{\overline}

\newcommand{\gr}{\bF_\theta^+}

\begin{document}
\title[Factoriality and Type Classification]
{Factoriality and type classification of $\textsf{k}$-graph von Neumann algebras}
\author[D. Yang]
{Dilian Yang}
\address{Dilian Yang,
Department of Mathematics $\&$ Statistics, University of Windsor, Windsor, ON
N9B 3P4, CANADA} \email{dyang@uwindsor.ca}

\begin{abstract}
Let $\Fth$ be a single vertex \textsf{k}-graph,
and $\pi_\omega(\O_\theta)''$ be the von Neumann algebra induced from the GNS representation of a distinguished
state $\omega$ of its $\textsf{k}$-graph C*-algebra $\O_\theta$.
In this paper, we prove the factoriality of $\pi_\omega(\O_\theta)''$ and further determine its type, 
when either $\Fth$ has the little pull-back property, or the intrinsic group of $\Fth$ has rank $0$.
 The key step to achieve this is to show that the fixed point algebra of the modular action 
corresponding to $\omega$ has a unique tracial state. 
\end{abstract}

\thanks{The research was supported in part by an NSERC Discovery grant.}

\date{}
\maketitle

\section{Introduction}

Since the work \cite{KumPask} of Kumjian and Pask, higher rank graph (or $\textsf{k}$-graph) algebras have been extensively studied. 
See, for example, \cite{FowSim, HRSW13, HLRS13, KumPask, KPS11, KPS12, PRRS06, Raeburn, RaeSimYee1,RaeSimYee2} to mention but a few and the references therein
for self-adjoint algebras, and \cite{DPY1, DYrepk, DPYdiln, FulYan13, KriPow06, P1} for non-self-adjoint algebras. Being a special case, 
single vertex $\textsf{k}$-graph algebras provide a very important and intriguing class \cite{DYrepk}. 
In particular, single vertex 2-graph algebras are systematically studied in \cite{DPY1, DPYdiln, DYperiod, FulYan13, Yang1, Yang2}.  

Roughly speaking, a single vertex \textsf{k}-graph is a unital semigroup $\Fth$ having \textsf{k} types of generators. 
The generators from the same type form a unital free semigroup, while the generators from distinct types satisfy a family $\theta$ of permutations. 
It is often useful to think of a \textsf{k}-graph as a \textsf{k}-coloured directed graph, whose type $i$-generators are edges with $i$-th colour ($1\le i\le \textsf{k}$). 
The graph C*-algebra of $\Fth$ is denoted by $\O_\theta$, which is the universal C*-algebra 
for the Cuntz-type representations of $\Fth$. 
The universal property of $\O_\theta$ yields a family of gauge automorphisms. Integrating those
automorphisms over the \textsf{k}-torus $\bT^\textsf{k}$ gives a faithful expectation $\Phi_ 0$ of $\O_\theta$
onto the core $\fF$ (i.e. the fixed point algebra of the gauge action) of $\O_\theta$. It is known that $\fF$ is a UHF algebra.
Particularly, $\fF$ has a unique tracial state, denoted by $\tau$. Then $\tau$
and $\Phi_ 0$ induce a distinguished
faithful state $\omega$ of $\O_\theta$, namely, $\omega:=\tau\circ \Phi_ 0$. Let $\pi_\omega(\O_\theta)''$
denote the von Neumann algebra induced from the GNS representation of $\omega$. 
Slightly abusing the terminology, we will simply call 
it the \textsf{k}-graph von Neumann algebra associated to $\Fth$ (induced from $\omega$). 
As in \cite{Yang1} for $\textsf{k}=2$, one can obtain an explicit formula for the modular automorphism $\sigma_t$ ($t\in \bR$) of $\pi_\omega(\O_\theta)''$ corresponding to $\omega$.
Identifying $\O_\theta$ with $\pi_\omega(\O_\theta)$ yields a C*-dynamical system $(\O_\theta, \bR, \sigma)$.

In \cite{Yang1, Yang2}, we initiate the study of the factoriality and type of $\pi_\omega(\O_\theta)''$ for a single vertex $2$-graph $\Fth$. 
Let $m$ and $n$ be the numbers of the blue and red edges of $\Fth$, respectively. 
It is shown in \cite{Yang1} that if $\frac{\ln m}{\ln n}\not\in\bQ$,
then $\pi_\omega(\O_\theta)''$ is a type III$_1$ (AFD) factor. 
In order to obtain the factoriality of $\pi_\omega(\O_\theta)''$ when $\frac{\ln m}{\ln n}\in\bQ$ for (aperiodic) 2-graph $\Fth$,
we carefully study the structure of the fixed point algebra $\O_\theta^\sigma$ of the modular action $\sigma$. 
It is proved in \cite{Yang2} that $\O_\theta^\sigma$ is a crossed product of the core $\fF$ by $\bZ$. Under the assumption that 
{$\O_\theta^\sigma$ has a unique tracial state, we are able to show that $\pi_\omega(\O_\theta)''$ 
is a type III$_\lambda$ factor, where $0<\lambda<1$ is completely determined by $m$ and $n$. 
It is asked in \cite{Yang2} if this assumption that $\O_\theta^\sigma$ has a unique tracial state is redundant. 

In this paper, 
we answer this affirmatively when $\Fth$ has the little pull-back property (see  Section \ref{SS:kgraph} below for its definition). 
Actually we are able to obtain much more.
Before going further, let us introduce an important group $G$ associated to $\Fth$:
\[
G=\big\{g=(g_1,\ldots, g_\textsf{k})\in\bZ^\textsf{k}\mid  \prod_{i=1}^\textsf{k}\, m_i^{g_i}=1\big\},
\]
where $m_i$ is the number of $i$-th coloured edges of $\Fth$ ($1\le i\le \textsf{k}$).
$G$ is called the \textit{intrinsic group} of $\Fth$. 
In general, $0\le \rank G\le \textsf{k}-1$.
In this paper, we prove that  $\pi_\omega(\O_\theta)''$ is a factor and further determine its type, 
when either $\Fth$ has the little pull-back property, or the intrinsic group of $\Fth$ has rank $0$.
To this end, we carefully study the structure of the fixed point algebra $\O_\theta^\sigma$,
and show that $\O_\theta^\sigma$ has the property of possessing a unique tracial state. 

The remainder of this paper is organized as follows. In Section \ref{S:pre}, some backgrounds needed later are briefly given. 
We first investigate the modular theory of \textsf{k}-graph algebras in Section \ref{S:stru}. Since it is similar to the case of $\textsf{k}=2$ studied in \cite{Yang1}, we 
sketch it here and only give somehow more unified proofs when necessary. 
Then we study the structure of the fixed point algebra $\O_\theta^\sigma$. 
It is proved that $\O_\theta^\sigma$ is a C*-algebra generated by the standard generators of $\O_\theta$ with degrees in $G$.
When $\rank G=1$, it is a crossed product of  $\fF$ by $\bZ$. 
In Section \ref{S:Dix}, we show that $\O_\theta^\sigma$ has the Dixmier property if $\Fth$ has the little pull-back property, and so it has a unique tracial state.
This is obtained as a consequence of the fact that $\O_\theta$ enjoys a property stronger than the Dixmier property. 
To this end, it needs some careful analysis about a class of averaging operators in $\O_\theta$. Those averaging operators are induced from
unitaries in the algebraic part of the core $\fF$. 
The idea of this part is strongly influenced by the work \cite{Arc80} of Archbold.
Our mains theorems are stated and proved in Section \ref{S:char}.

We should mention that after the first version of this paper was circulated in 2013, very recently the main results of this paper have been generalized in 
\cite{LLNSW15} by using completely different approaches. 

\smallskip
Let us finish off this introduction by fixing our notation and conventions. 

\smallskip
\noindent
\textbf{Notation and Conventions.}
There is a natural coordinate-wise partial ordering on $\bZ^\textsf{k}$: 
$
p,  q\in\bZ^\textsf{k} \text{ with } p\le  q \Longleftrightarrow p_i\le q_i \text{ for all } 1\le i\le \textsf{k}. 
$
For $ p,  q\in\bZ^\textsf{k}$, $p\vee  q$ and $ p\wedge  q$ denote the coordinate-wise maximum and minimum of $ p$ and $ q$, respectively.  
Let $ p_+= p\vee 0$, and $ p_-=(- p)\vee  0$. 
Set $\bN^\textsf{k}=\{ p\in\bZ^\textsf{k}:  p\ge  0\}$. 
Then $p= p_+- p_-$ with $ p_+, p_-$ in $\bN^\textsf{k}$ and  
$ p_+\wedge p_-= 0$. 

For a nonzero $n\in\bN$ 
we use the following notation: 
\begin{align*}
\ol{n}=(n,\ldots, n)\in\bN^\textsf{k},\quad 
 \bn=\{1,\ldots, n\}.
\end{align*}

In this paper, every C*-algebra $\A$ is assumed to be unital. The group of its unitaries is denoted by $\U(\A)$, and its centre is written as $\Z(\A)$.  
By an endomorphism of $\A$, we always mean a unital, *-endomorphism. 

Let $\ep_i$ ($1 \le i \le \textsf{k}$) be the standard generators for $\bZ^\textsf{k}$.
From now on, we always assume that 
 \textit{$\Fth$ is a fixed single vertex $\textsf{k}$-graph with $m_i$ $i$-th coloured generators (i.e., of degree $\epsilon_i$) ($1\le i\le \textsf{k})$.}

If $\Fth$ is periodic (equivalently, $\O_\theta$ is not simple), then it is known that the centre of $\O_\theta$ is isomorphic to $\rC(\bT^s)$ for some $s\ge 1$
\cite{DYrepk, DYperiod}. So $\pi_\omega(\O_\theta)''$ is not a factor in this case.
Hence, throughout the rest of this paper, we always assume that $\Fth$ is aperiodic, unless otherwise specified. 
In particular, this implies that  $m_i>1$ for all $1\le i\le \textsf{k}$ \cite{DYrepk}.  
Set 
\[
m:=(m_1, \ldots, m_\textsf{k}),
\]
and we use the multi-index notation: for $p\in\bZ^\textsf{k}$, 
$
m^{ p}:=\prod_{i=1}^\textsf{k} m_i^{p_i}.
$

Finally, if $p\in\bN^\textsf{k}$, the set of words in $\Fth$ of degree $p$ is denoted by 
$
\Lambda^{ p}=\{w\in\Fth\mid d(w)= p\}. 
$


\section{Preliminaries}
\label{S:pre}

The main sources of this section are \cite{Arc80, DPY1, DYrepk, Exe08, KumPask, Raeburn, Yang1}. 

\subsection{$\textsf{k}$-graphs}
\label{SS:kgraph}

Kumjian and Pask \cite{KumPask} define a $\textsf{k}$-graph as a small category
$\Lambda$ with a degree map $d : \Lambda\to\bN^\textsf{k}$ satisfying the
\textit{factorization property}: for every $\lambda\in\Lambda$
and $p,q\in\bN^{\textsf{k}}$ with $d (\lambda)=p+q$, there are unique
elements $\mu,\nu\in\Lambda$ such that
$\lambda=\mu\nu$ with $d(\mu)=p$ and $d(\nu)=q$.

Recall that a $\textsf{k}$-graph $\Lambda$ is \textit{finitely aligned} if for every $\mu, \nu\in\Lambda$ one has
\[
\Lambda^{\text{min}}(\mu,\nu):=\{(\xi,\eta)\in\Lambda\times\Lambda:\mu\xi=\nu\eta\text{ and }d(\mu\xi)=d(\mu)\vee d(\nu)\}
\]
is finite. 

\begin{defn}
We say that $\Lambda$ is \textit{singly aligned}, if $|\Lambda^{\text{min}}(\mu,\nu)|\le 1$ for every $(\mu,\nu)\in\Lambda\times\Lambda$.
\end{defn}

An obvious weaker notion is the following.

\begin{defn}
A \textsf{k}-graph $\Lambda$ is said to have the \textit{little pull-back property} 
if $|\Lambda^{\text{min}}(e,f)|\le 1$, whenever $e,f$ are edges of $\Lambda$ with distinct degrees.
\end{defn}

To my best knowledge, the little pull-back property was first introduced by Exel in \cite{Exe08}  
in order to apply his tight representation theory
to $\textsf{k}$-graphs.
This property has nice and clear geometrical meaning in the commuting-square representation. 
It turns out that that the little pull-back property and single alignment are equivalent (\cite{Exe08}).  

\medskip
In this paper, we focus on single vertex $\textsf{k}$-graphs. 
Their definition can be equivalently stated as follows.
\begin{defn}
\label{D:kgraph}
A single vertex \textsf{k}-graph $\Fth$ is a unital semigroup generated by $\{e_\fs^i:\fs\in \Bm_i,\, 1\le i\le \textsf{k}\}$, which satisfies the following properties:
\begin{itemize}
\item[(\textsf{fsg}).] $\{e_\fs^i:\fs\in \Bm_i\}$ generates a unital free semigroup $\bF_{m_i}^+$;
\item[(\textsf{$\theta$-CR}).] There is a family of permutations $\theta=\{\theta_{ij}: \theta_{ij}\in S_{ \bm_i \times  \bm_j} \text{ for } 1\le i<j\le \textsf{k}\}$, such that 
the following $\theta$-commutation relations hold
\begin{align*}
 e^i_\fs e^j_\ft = e^j_{\ft'} e^i_{\fs'}
 \quad\text{where}\quad
 \theta_{ij}(\fs,\ft) = (\fs',\ft').
\end{align*}
\item[(\textsf{Cubic}).] Every three sets of generators $\{e_{t_1}^i, e_{t_2}^j, e_{t_3}^k\}$ satisfy the following
cubic condition:
\begin{align*}
e^i_{\ft_1} e^j_{\ft_2} e^k_{\ft_3}
&= e^i_{\ft_1} e^k_{\ft_3'} e^j_{\ft_2'}
= e^k_{\ft_3''} e^i_{\ft_1'} e^j_{\ft_2'}
= e^k_{\ft_3''} e^j_{\ft_2''} e^i_{\ft_1''} \\
&= e^j_{\ft_{2'}} e^i_{\ft_{1'}} e^k_{\ft_3}
= e^j_{\ft_{2'}} e^k_{\ft_{3'}} e^i_{\ft_{1''}}
= e^k_{\ft_{3''}} e^j_{\ft_{2''}} e^i_{\ft_{1''}}
\end{align*}
\[\implies \quad
e^i_{\ft_1''}=e^i_{\ft_{1''}},\
e^j_{\ft_2''}=e^j_{\ft_{2''}},\
e^k_{\ft_3''}=e^k_{\ft_{3''}}.
\]
\end{itemize}
\end{defn}

By the above definition, for $\textsf{k}=2$, every permutation completely determines a 2-graph. 
But for $\textsf{k}\ge 3$, not every family $\theta$ of permutations yields a $\textsf{k}$-graph.
The above cubic condition requires, very roughly speaking,  that a word of degree $(1,1,1)$ should be well defined.
Thus $\Fth$ is a $\textsf{k}$-graph if and only if the restriction of $\Fth$
to every triple family of edges $\{ e^i_\fs,\ e^j_\ft,\ e^k_\fu \}$ is a 3-graph.

Return to the little pull-back property. 
One can see that $\Fth$ has the little pull-back property if and only if so is every 
$2$-graph $\{ e^i_\fs,\ e^j_\ft\}$. By \cite{DYrepk}, if $\Fth$ has the little pull-back property, then it has to be aperiodic; 
however, the converse does not hold in general. 
The little pull-back property seems to be restrictive. But it is actually not very so. It is known in \cite{P1} that there are essentially only nine single
vertex 2-graphs with $2$ blue and $2$ red edges. One can easily check that all aperiodic graphs have the little 
pull-back property (and so are singly aligned) except only one(!)
given by $\theta=((1,1),(2,1),(1,2))$.

\subsection{\textsf{k}-graph algebras}
Recall from \cite{FulYan13} that an $n$-tuple $V=(V_1,...,V_n)$ of operators in a C*-algebra $\A$ is said to be of \textit{Cuntz-type} if 
\[
V_i^*V_j=\delta_{i,j} \quad\text{and}\quad \sum_{i=1}^nV_iV_i^*=I.
\]
Namely, $V_i$'s are orthogonal isometries, and satisfy the defect free property in the terminology of \cite{DPY1, DPYdiln, DYperiod, DYrepk}.

Let $\gr$ be a $\textsf{k}$-graph, and $E^i=(E^i_1, \ldots, E^i_{m_i})$ $(1\le i\le \textsf{k}$). 
We call $[E^1, \ldots, E^\textsf{k}]$ a \textit{representation of $\Fth$} if it satisfies ($\theta$-CR):
\[
 E^i_\fs  E^j_\ft= E^j_{\ft'} E^i_{\fs'} \quad \text{when}\quad \theta_{ij}(\fs,\ft)=(\fs',\ft')
\]
for all $1\le i< j\le \textsf{k}$, $\fs\in  \Bm_i$, $\ft\in  \Bm_j$. 
This representation is simply denoted as $\textsf{E}$. $\textsf{E}$ is said to be of \textit{Cuntz-type}, if every $ E^i$ is of Cuntz-type ($1\le i\le \textsf{k}$).

\begin{defn}
The $\textsf{k}$-graph C*-algebra of $\Fth$, denoted as $\O_\theta$, is the universal C*-algebra for Cuntz-type representations of $\Fth$.
\end{defn}

So $\O_\theta$ is the C*-algebra generated by $\Fth$ with the universal property that every Cuntz-type representation of $\Fth$ extends uniquely 
a Cuntz-type representation of $\O_\theta$. 
Also, we reserve the notation $s_{e^j_\ft}$, simply written as $s^j_\ft$ ($1\le j\le \textsf{k},\, \ft\in \Bm_j$), for the generators of 
$\O_\theta$. 

For $w$ in $\Fth$, by $\theta$-commutation relations ($\theta$-CR), one can always write $w=u_1\cdots u_\textsf{k}$ with $u_i$ being a word of $e^i$'s. 
The degree of $w$ is defined by 
$d(w)=(|u_1|, \ldots, |u_\textsf{k}|)$, 
where $|u_i|$ is the length of $u_i$. 
For $w=e^{j_1}_{\ft_1}\cdots e^{j_n}_{\ft_n}\in\Fth$, we use the multi-index notation
\[
s_w=s_{e^{j_1}_{\ft_1}}\cdots s_{e^{j_n}_{\ft_n}}=s^{j_1}_{\ft_1}\cdots s^{j_n}_{\ft_n}.
\]
The degree map $d$ of $\Fth$ can be extended as follows.
By the universal property of $\O_{\theta}$, there is a family of
\textit{gauge automorphisms}
$\gamma_{\bt}$ for $\bt \in \bT^\textsf{k}$
given by
\[
 \gamma_{\bt}(s_w) =\bt^{d(w)} s_w\qforal w\in\Fth.
\]
For each $ n\in\bZ^\textsf{k}$, define a mapping $\Phi_ n$ on $\O_{\theta}$ via
\[
\Phi_ n(A)=\int_{\bT^\textsf{k}}\bt^{- n}\gamma_\bt(A) d\bt \quad \mbox{for all}\quad A\in\O_{\theta}.
\]
So $\ran\Phi_n$ is a spectral subspace of the gauge action $\gamma$. 
If $A\in\ran\Phi_ n$, we say that the \textit{degree} of $A$ is $ n$, and write $d(A)= n$. So in particular $d(s_us_v^*)=d(u)-d(v)$ for $u, v\in\Fth$.

It turns out that $\Phi_ 0$ is of particular importance to us. It is a faithful expectation onto the core $\O_{\theta}^\gamma$.
Furthermore, $\O_\theta^\gamma$ 
is a $\prod_{i=1}^\textsf{k}m_i^\infty$-UHF algebra:
\[
\fF:=\ran\Phi_ 0=\O_{\theta}^\gamma=\ol{\bigcup_{n\ge 1} \fF_n},
\]
where
\[
\fF_n=\spn\{ s_us_v^* : d(u)=d(v)=\ol n\}
\]
is the full matrix algebra $M_{\prod_{i=1}^\textsf{k}m_i^n}$.
(See \cite{Yang1} for the explanation why one can only consider such special elements $\ol{n}$ in $\bN^\textsf{k}$.) 
Let $\tau$ be the (unique) tracial state of $\fF$.
With $\Phi_ 0$, $\tau$ induces a distinguished faithful state $\omega$ on $\O_\theta$ by
$$
\omega(A)=\tau(\Phi_ 0(A))\quad \text{for all}\quad A\in\O_\theta.
$$
Let $\pi_{\omega}(\O_\theta)''$ be the von Neumann algebra generated by the GNS representation of
$\omega$. When no confusion is caused, we shall omit the subscript $\omega$ and write $\pi(\O_\theta)''$ for short. 
Abusing the terminology, we give the following definition. 

\begin{defn}
The von Neumann algebra $\pi(\O_\theta)''$ is called the \textit{$\textsf{k}$-graph von Neumann algebra} associated to $\Fth$ (induced from $\omega$). 
\end{defn}

Also, since $\pi$ is faithful, we will often identify $\O_\theta$ and $\pi(\O_\theta)$. 
Refer to \cite{DYrepk, Yang1} and the references therein for more details about this subsection.

\subsection{Endomorphisms of C*-algebras arising from Cuntz-type tuples}
Let $\A$ be a C*-algebra. 
Let $V=(V_1,\ldots, V_n)$ be an arbitrary Cuntz-type tuple in $\A$ and $\ca(V)$ be the C*-algebra generated by $V_i$'s. 
For $p\in\bN$, one has a Cuntz-type $n^p$-tuple:
$V_p=(V_w: w\in \Fn,  |w|=p).$ 
As in \cite{Yang1},  we define 
\[
 \gamma_p(A)=\sum_{|w|=p}  V_w A  V_w^*\qforal A\in \ca(V).
\]
Then $ \gamma_p$ is an endomorphism of $\ca(V)$. Clearly, $ \gamma_p$ can be extended to $\A$, which is also denoted by $ \gamma_p$, 
sometimes more precisely by $ \gamma_p^V$ in order to emphasize the tuple $V$ which gives rise to it. 

Return to the $\textsf{k}$-graph C*-algebra $\O_\theta$. 
Let $ \textsf{E}=[ E^1,...,  E^\textsf{k}]$ be a Cuntz-type representation of $\O_\theta$. One obtains an endomorphism $ \gamma_ p$
 on $\ca(\textsf{E})$ for $ p\in\bN^\textsf{k}$. If $E^i$'s are from $\O_\theta$, then we have an endomorphism $ \gamma^\textsf{E}_ p$ on $\O_\theta$:
\begin{align}
\label{E:genGam}
 \gamma^{ \textsf{E}}_ p(A)=\sum_{w\in\Lambda^p}  \textsf{E}_w A  \textsf{E}_w^*\qforal A\in\O_\theta.
\end{align}
Here the multi-index notation $\textsf{E}_w$ is similar to $s_w$ as above.

\subsection{The Dixmier property}
\label{SS:DP}

Let $\A$ be a C*-algebra and $\V$ be a subgroup of the unitary group $\U(\A)$ of $\A$. For given $U_1,...,U_n\in\V$ and positive numbers 
$\lambda_1,...,\lambda_n$ with $\sum_{i=1}^n\lambda_i=1$, we define an averaging operator $\alpha$ on $\A$ via
\[
 \alpha(A)=\sum_{i=1}^n\lambda_i U_i A U_i^*\qforal A\in \A.
\]
We let $\textsf{Ave}(\A, \V)$ denote the set of all such averaging operators. It is easy to see that $\textsf{Ave}(\A,\V)$ is a unital semigroup with composition as 
the semigroup multiplication, and that every $\alpha\in\textsf{Ave}(\A,\V)$ is a unital positive map. In particular, every averaging operator is contractive and self-adjoint.

\begin{defn}
\label{D:Dix}
A C*-algebra $\A$ is said to have the Dixmier property if, for every $A\in\A$, 
\[
\ol{\{\alpha(A):\alpha\in \textsf{Ave}(\A,\U(\A))\}}^{\|\cdot \|}\cap \Z(\A)\ne \mt. 
\] 

We say that $\A$ satisfies the \text{weak Dixmier property} if, for every $A\in\A^+\setminus\{0\}$,
\[
\ol{\spn\{\alpha(A):\alpha\in \textsf{Ave}(\A,\U(\A))\}}^{\|\cdot \|}\supseteq \bC I. 
\]
\end{defn}

The Dixmier property has been extensively studied. See, for example, \cite{Arc77, Arc79, Arc80, HaaZsi84, Rie82}, amongst others. 
It is well-known that a simple C*-algebra with the Dixmier property has at most one tracial state. 
Haagerup and Zsid\'o \cite{HaaZsi84} showed the converse for simple C*-algebras:
every simple C*-algebra with at most one tracial state has the Dixmier property.   
The weak Dixmier property is introduced in \cite{Rie82} in order to partially answer a question posed in \cite{Arc79}. 
Clearly, the weak Dixmier property is weaker than the Dixmier property. 
It is proved in \cite{Rie82} that $\A$ satisfies the weak Dixmier property if and only if $\A$ is simple and has a most one tracial state. 
So it is now apparent that the Dixmier property and the weak Dixmier property are equivalent for simple C*-algebras. 
One advantage of the weak Dixmier property is that it implies simplicity.

\section{The modular theory of $\O_\theta$ and Structure of $\O_\theta^\sigma$}
\label{S:stru}

In this section, we first briefly discuss the modular theory of \textsf{k}-graph C*-algebras. This is very similar to \cite{Yang1} for $\textsf{k}=2$. 
The associated modular objects are explicitly given.
They will be useful in calculating the Connes spectrum of the modular operator. In the second part of this section, 
we investigate the structure of $\O_\theta^\sigma$, the fixed point algebra of the modular action corresponding to the 
distinguished state $\omega$. Its structure plays a critical role for what follows. 

\subsection{The modular theory of $\O_\theta$}

Let ${{}^\circ\O_{\theta}}$ denote the \textit{algebraic} algebra generated by $s_u$'s.
So ${}^\circ{\O_\theta}=\spn\{ s_u s_v^*:u,v\in\Fth\}$.
In the sequel, for $ n\in\bZ^\textsf{k}$, we use $\X_ n$ to denote the intersection of $\ran{\Phi_ n}$ and ${}^\circ\O_{\theta}$: 
\[
\X_ n=\spn\{ s_u s_v^*:u,v\in\Fth, \, d(u)-d(v)= n\}.
\]
So every element of $\X_n$ has degree $n$. 

Recall from Section \ref{S:pre} that the distinguished faithful state $\omega$ is given by $\omega=\tau\circ\Phi_ 0$,
where $\tau$ is the tracial state on $\fF$ and $\Phi_ 0$ is the expectation of $\O_\theta$ onto the core $\fF$ of $\O_\theta$.
Let $L^2(\O_\theta)$ be the GNS Hilbert space determined by $\omega$. So the inner
product on $L^2(\O_\theta)$ is given by $\langle A\mid B\rangle=\omega(A^*B)$ for all $A,B\in L^2(\O_\theta)$. For $A\in\O_\theta$, we denote
the left action of $A$ by $\pi_\omega(A)$, that is, $\pi_\omega(A)B=AB$ for all $B$ in $\O_\theta$.

We now begin to give the modular objects in the celebrated Tomita-Takesaki modular theory.  
Define an anti-linear operator $S_\circ$ on ${}^\circ{\O_\theta}\subset L^2(\O_\theta)$ by
$$
S_\circ(A)=A^*\quad \mbox{for all}\quad A\in {}^\circ{\O_\theta},
$$
and another one $F_\circ$ on ${}^\circ{\O_\theta}$, which acts on generators by
$$
F_\circ( s_u s_v^*)
= m^{d(u)-d(v)} s_v s_u^*,
$$
and then extend it anti-linearly.

As in \cite{Yang1}, we shall show that $F_\circ$ is indeed the adjoint of $S_\circ$. The following proof is essentially the same as that of \cite[Lemma 5.2]{Yang1}, 
but in a slightly more unified way (even in the case of $\textsf{k}=2$).  
A lemma first, whose proof can be easily modified from that of \cite[Lemma 5.1]{Yang1} and so we omit it. 

\begin{lem}\label{L:tronF}
Suppose $u,v\in\Fth$ with $d(u)=d(v)$. Then
$$
\tau(s_u A s_v^*)
=\delta_{u,v}\, m^{-d(u)}\tau(A)
\quad \text{for all}\quad A\in\fF.
$$
\end{lem}

\begin{lem}
\label{L:FadjS}
Let $S_\circ$ and $F_\circ$ be defined as above.
Then $F_\circ=S_\circ^*$: $\langle S_\circ(A)\mid B\rangle=\langle F_\circ(B)\mid A\rangle$
for all $A,B$ in ${}^\circ{\O_\theta}$.
\end{lem}

\begin{proof}
It suffices to prove this lemma for standard generators $A, B$ in ${}^\circ{\O_\theta}$. By the definitions of $S_\circ$ and $F_\circ$, we have 
\[
\langle S_\circ(A)\mid B\rangle=\omega(AB), \quad \langle F_\circ(B)\mid A\rangle= m^{d(B)}\omega(BA).
\]

If $d(A)+d(B)\ne  0$, then neither $AB$ nor $BA$ belongs to $\fF$. So $\langle S_\circ(A)\mid B\rangle=\langle F_\circ(B)\mid A\rangle=0$. 

We now assume that $d(A)+d(B)= 0$.
Assume $d(A)=n\in \bZ^\textsf{k}$. 
Then, by the $\theta$-commutation relations, there are $u_1,v_1,u_2,v_2\in\Fth$ with 
$d(u_1)=d(u_2)= n_+$ and $d(v_1)=d(v_2)= n_-$, 
and $A',B'$ in $\fF$ such that
\[
A= s_{u_1}A'  s_{v_1}^*,\quad B= s_{v_2}B'  s_{u_2}^*.
\]  
Then by Lemma \ref{L:tronF} we have 
\begin{align*}
\omega(AB)
&=\omega( s_{u_1}A'  s_{v_1}^* s_{v_2}B'  s_{u_2}^*)=\delta_{v_1,\, v_2} \omega( s_{u_1}A'B'  s_{u_2}^*)\\
&=\delta_{v_1, v_2}\delta_{u_1,u_2}\,  m^{d(u_1)}\omega(A' B'),\\
\omega(BA)&=\omega( s_{v_2}B'  s_{u_2}^* s_{u_1}A'  s_{v_1}^*)=\delta_{u_1,\, u_2} \omega( s_{v_2}B'A'  s_{v_1}^*)\\
&=\delta_{u_1,u_2}\delta_{v_1,v_2}\,  m^{d(v_1)}\omega(B' A').
\end{align*}
Thus 
\[
\omega(AB)=  m^{d(v_1)-d(u_1)}\omega(BA)=  m^{d(v_2)-d(u_2)}\omega(BA)=  m^{d(B)}\omega(BA).
\]
This ends the proof.
\end{proof}

By Lemma \ref{L:FadjS}, we particularly obtain that both $F_\circ$ and $S_\circ$ are closable. 
We use $F$ and $S$ to denote their corresponding closures. The closure $S$ is called the \textit{Tomita operator}.
The following is well-known:
$S$ and $F$ have polar decompositions
\begin{eqnarray*}
S=J\Delta^{\frac{1}{2}}=\Delta^{-\frac{1}{2}}J,\quad
F=J\Delta^{-\frac{1}{2}}=\Delta^{\frac{1}{2}}J,
\end{eqnarray*}
where  $\Delta=FS$ is the \textit{modular operator}, and $J$ is the \textit{modular conjugation}.
By Lemma \ref{L:FadjS} we have
$$
J( s_u s_v^*)
= m^{\frac{d(u)-d(v)}{2}} s_v s_u^*,
$$
and for $z\in\bC$
$$
\Delta^z( s_u s_v^*)
= m^{z(d(v)-d(u))} s_u s_v^*.
$$

Some obvious modifications of the proof of \cite[Theorem 5.3]{Yang1} shows that 
the algebra ${}^\circ{\O_\theta}$, with the inner product $\langle\cdot\mid  \cdot \rangle$:
$\langle A\mid B\rangle=\omega(A^*B)$,
is a modular Hilbert algebra. Note that $\pi_\omega(\O_\theta)''$ is nothing but the left von Neumann algebra of 
${}^\circ\O_{\theta}$ (\cite{Take}).
The Tomita-Takesaki modular
theory says that
$$
\Delta^{it}\pi(\O_\theta)''\Delta^{-it}=\pi(\O_\theta)''\qforal t\in\bR.
$$
For $z\in\bC$, let 
\[
\sigma_{z}(\pi(A))=\Delta^{iz}\pi(A)\Delta^{-iz} \quad \text{for all}\quad A\in{}^\circ\O_{\theta}.
\]
The one-parameter group $\{\sigma_t:t\in\bR\}$ is called the \textit{modular automorphism group}.
As in \cite{Yang1}, it is not hard to obtain the formula of $\sigma_t$
on generators:
\begin{align}
\label{E:modaut}
\sigma_t(\pi(s_us_v^*))
= m^{it(d(v)-d(u))}\pi(s_us_v^*).
\end{align}
Refer to \cite{KadRin, Str, Take} for more information on the Tomita-Takesaki modular theory.

Identifying $\O_\theta$ with $\pi(\O_\theta)$, one obtains a C*-dynamical system $(\O_\theta,\bR, \sigma)$, or simply $(\O_\theta, \sigma)$.
The main purpose of the next subsection is to look closer at the structure of the fixed point algebra $\O_\theta^\sigma$.

\subsection{The structure of $\O_\theta^\sigma$}
Recall that $m=(m_1,\ldots, m_\textsf{k})$, where, for $1\le i\le \textsf{k}$, $m_i$ is the number of edges of $\Fth$ of degree $\epsilon_i$. 
Define
\begin{align}
\label{E:groupG}
G=\{ g\in\bZ^\textsf{k}:  m^ g=1\}.
\end{align}
It is easy to check that $G$ is a subgroup of $\bZ^\textsf{k}$.
This group is important in this section and in classifying $\textsf{k}$-graph von Neumann algebras later. 
So it deserves having its own name. 

\begin{defn}
\label{D:groupG}
The group $G$ defined in \eqref{E:groupG} is called the intrinsic group of $\Fth$. 
\end{defn}

In the following, let $r$ be the rank of $G$ and $ g_1, ...,  g_r$ be the generators of $G$. 

Consider the set $\{\ln m_i: 1\le i \le \textsf{k}\}$. 
There are two cases for this set: it is either rationally dependent or independent.  

Suppose first that it is rationally dependent. Then it is equivalent to 
$ m^ g=1$ for some $ g$ in $\bZ^\textsf{k}\setminus\{ 0\}$.
Let us notice that for every $g\in G\setminus\{ 0\}$, one has
$ g_+\ne  0$ and $ g_-\ne  0$.

For $1\le p\le r$, consider the generator $ g_p$. 
Since $ g_p\in G$, we have 
\[
m^{{ g_p}_+}= m^{{ g_p}_-}.
\]
So there is a bijection, say $\jmath_p$ (once chosen and fixed), from the subset 
\[
\Lambda_{ {{g_p}_+}}=\{w\in\Fth: d(w)={ g_p}_+\}
\]
onto the subset 
\[
\Lambda_{ {{g_p}_-}}=\{w\in\Fth: d(w)={ g_p}_- \}.
\]
Define an operator $U_p$ in $\O_\theta$ by
\[
U_p=\sum_{u\in\Lambda_{{ g_p}_+}} s_us_{{\jmath_p}(u)}^*,
\]
which is easily seen to be unitary. 

We now define an action $\rho$ of the free group $\bF_r$ with the generators $a_1,...,a_r$ on $\fF$ as follows:
\[
\rho_{a_p}=\Ad_{U_p}\quad (1\le p\le r).
\]
Then we obtain a C*-dynamical system $(\fF, \bF_r,\rho)$. 

\smallskip
If $\{\ln m_i: 1\le i \le \textsf{k}\}$ is rationally independent, equivalently, $G=\{ 0\}$ and so $r=0$, , we make the convention that $U_0=I$.

\smallskip
The following two results generalize \cite[Theorem 3.1]{Yang2} where $G=\bZ(a,-b)$, and unifies some proofs of \cite{Yang1, Yang2}. 

\begin{prop}
\label{P:Ofixed}
Keep the same notation as above. Then 
\[
\O_\theta^\sigma=\ca(X\in \O_\theta: d(X)\in G)=\ca(\fF, U_1, \ldots, U_r). 
\] 
\end{prop}

\begin{proof}
Similar to the proof of \cite[Theorem 3.1]{Yang2}, 
to show the first ``=", it suffices to verify $\O_\theta^\sigma\subseteq \ca(X\in \O_\theta: d(X)\in G)$. To this end, let us consider the generator $X= s_u s_v^*$
in $\O_\theta^\sigma$. Since $\sigma_t(X)=X$ for all $t\in \bR$, one has $e^{it\ln  m^{d(v)-d(u)}}=1$ for all $t\in \bR$. This implies $ m^{d(v)-d(u)}=1$, namely, $d(v)-d(u)\in G$, as required. 

To prove the second ``=", one needs to check $\ca(X\in \O_\theta: d(X)\in G)\subseteq\ca(\fF, U_1, \ldots, U_r)$. For this, take $X\in \O_\theta$ with $d(X)\in G$ and 
without loss of generality assume 
that $X= s_u s_v^*$. Rewrite it `in the order of $ a_1, \ldots,  a_r$', and then use the definition of $U_p$'s to obtain 
\begin{align*}
X
&= s_{u_{1,1}}\cdots  s_{u_{1,n_1}} \cdots  s_{u_{r,1}}\cdots  s_{u_{r, n_r}}   s_v^*\\
&=U_1 s_{\jmath_1{(u_{1,1}})}\cdots U_1  s_{\jmath_1(u_{1,n_1})} \cdots U_r s_{\jmath_r(u_{r,1})}\cdots U_r s_{\jmath_r(u_{r, n_r})}   s_v^*\\  
&=U_1 s_{\jmath_1{(u_{1,1}})}\cdots U_1  s_{\jmath_1(u_{1,n_1})} \cdots U_r s_{\jmath_r(u_{r,1})}\cdots U_r\, s_{\jmath_r(u_{r, n_r})}  \cdot  s_v^*\cdot
 (U_r^{*n_r}\cdots U_1^{*n_1})\\
&\quad \cdot  (U_1^{n_1}\cdots U_r^{n_r})\\  
&\in \fF\,  U_1^{n_1}\cdots U_r^{n_r},
\end{align*}
where $u_{p,j}\in \Lambda_{{ g_p}_+}$ for all $1\le p\le r$ and $1\le j\le n_p$. 
\end{proof}

With the aid of \cite{Hop05}, the following lemma can be easily proved as \cite[Lemma 4.1]{Yang1} where $\textsf{k}=2$ , and so we omit its proof here. 

\begin{lem}
\label{L:masa}
The following statements are equivalent.
\begin{itemize}
\item[(i)] $\Fth$ is aperiodic.
\item[(ii)] The relative commutant of $\fF$ is trivial: $\fF'\cap\O_\theta=\bC$.
\end{itemize}
\end{lem}

\begin{prop}
\label{P:fixpointalg}
Suppose $\rank G=1$.  
Then $\O_\theta^\sigma$ is a simple C*-algebra, which is isomorphic to the crossed product  of $\fF$ by $\bZ$: 
\[
\O_\theta^\sigma\cong \fF\rtimes_{\rho} \bZ.
\]
\end{prop}

\begin{proof}
If $r=0$, namely, $\{\ln m_i: 1\le i \le \textsf{k}\}$ is rationally independent, then $\O_\theta^\sigma$ coincides with $\fF$ by Proposition \ref{P:Ofixed} and so clearly has all required properties. 

Suppose $r=1$.  
Since $\Fth$ is aperiodic, as shown in \cite[Theorem 3.1]{Yang2} using Lemma \ref{L:masa}, $\rho$ is aperiodic meaning that $\rho^n$ is outer for
every $n\in \bZ$. Then it follows from,
e.g. \cite{Kish}, that the reduced crossed product $\fF\rtimes_{\rho,\, \text{r}}\bZ$ is simple. But since $\bZ$ is amenable, one has 
$\fF\rtimes_{\rho}\bZ\cong\fF\rtimes_{\rho,\, \text{r}}\bZ$. 
So the universal property of $\fF\rtimes_\rho \bZ$ gives the required isomorphism by Proposition \ref{P:Ofixed}.
\end{proof}

It would be interesting to know the dynamical system structure of $\O_\theta^\sigma$ when $\rank G>1$.

\section{The Dixmier Property of $\O_\theta^\sigma$}
\label{S:Dix}

In \cite{Arc80}, Archbold studied the Dixmier property of Cuntz algebras $\O_n$. But actually 
more was obtained there:\footnote{This was kindly informed to the author by Professor Robert Archbold.}
the averaging operators $\alpha$ in Definition \ref{D:Dix} can be chosen from $\textsf{Ave}(\O_n, \U(X_0))$, 
where $X_0$ is the algebraic part of the core of $\O_n$.
 It is this stronger property of $\O_n$ that motivates this section.  
The aim of this section is to show that the fixed point algebra $\O_\theta^\sigma$ has a unique tracial state if 
either $\Fth$ has the little pull-back property, or the intrinsic group of $\Fth$ has rank $0$. 
To achieve this, we first prove that $\O_\theta^\sigma$ is simple and has the Dixmier property. 

However, the higher rank flavour of \textsf{k}-graphs causes some complication here. For instance, the degree of a standard generator of $\O_n$
and $0$ are always comparable; but this is not the case for \textsf{k}-graphs. So it seems that some key results 
in \cite{Arc80} can not be generalized to \textsf{k}-graph algebras.
Fortunately, in our case, we can make full use of the little pull-back property of $\Fth$  
and the defect free property of Cuntz-type representations to achieve our goal.

Let $\textsf{E}$ be a Cuntz-type representation of $\Fth$. 
For $ p\in \bN^\textsf{k}$, let $\Omega_{ p}$ be the set of all functions from $\Lambda^{ p}$, the set of all words in $\Fth$ 
with degree $ p$, to $\{0,1\}$, and $\P_{ p}$ be the permutation group of $\Lambda^{ p}$.
For $f\in \Omega_{{ p}}$ and $\varsigma\in \P_{{ p}}$, define the unitary operator by 
\[
U_{(f,\varsigma)}=\sum_{w\in\Lambda^{ p}}(-1)^{f(w)}  \textsf{E}_{\varsigma(w)} \textsf{E}_w^*,
\]
and the averaging operator on $\ca(\textsf{E})$ via
\begin{align}
\label{E:gam}
\alpha_{ p}(A)
=\frac{1}{2^{{ m}^{ p}}({ m}^{ p})!} \sum_{f\in\Omega_{{ p}}, \, \varsigma\in\P_{{ p}}}U_{(f,\varsigma)} A U_{(f,\varsigma)}^*\qforal A\in \ca(\textsf{E}).
\end{align}
As $ \gamma^\textsf{E}_ p$ in Section \ref{S:pre}, when the operators of $\textsf{E}$ are from $\O_\theta$,  we use $\alpha^\textsf{E}_ p$ to denote its extension to $\O_\theta$.

\subsection{Some auxiliary lemmas}
In this subsection, we prove some technical lemmas, which will be useful later. 

\begin{lem}
\label{L:2basic}
Let $\textsf{E}$ be a Cuntz-type representation of $\O_\theta$ and $p\in\bN^\textsf{k}$. 
Suppose that $u,v\in\Fth$ are such that $u=u_1u_2$ and $v=v_1v_2$ with $u_i$, $v_i$ in $\Fth$ ($i=1,2$)
and $d(u_1)=d(v_1)=p$. 
Then
\[
\alpha_ p( \textsf{E}_u \textsf{E}_v^*)
=\frac{\delta_{u_1,v_1}}{m^p} \gamma_ p( \textsf{E}_{u_2})\gamma_p(\textsf{E}_{v_2})^*.
\]
\end{lem}

\begin{proof}
The following proof is essentially the same as that of \cite[Proposition 2 (ii)]{Arc80}, so we only sketch it here.

Indeed, we have 
\begin{align*}
\nonumber
m^{ p}\, \alpha_ p( \textsf{E}_u \textsf{E}_v^*)
&=\frac{1}{2^{{ m}^{ p}}({ m}^{ p}-1)!}\sum_{f\in\Omega_ p,\,\varsigma\in\P_ p} \left(\sum_{\mu\in\Lambda^{ p}} (-1)^{f(\mu)} \textsf{E}_{\varsigma(\mu)} \textsf{E}_\mu^*\right) \cdot \\
\nonumber
&\qquad\qquad\quad \textsf{E}_{u_1} \textsf{E}_{u_2} \textsf{E}_{v_2}^* \textsf{E}_{v_1}^*
\cdot \left(\sum_{\nu\in\Lambda^{ p}} (-1)^{f(\nu)} \textsf{E}_\nu \textsf{E}_{\varsigma(\nu)}^*\right) \\
\nonumber
&=\frac{1}{2^{{ m}^{ p}}({ m}^{ p}-1)!}\sum_{f\in\Omega_ p,\,\varsigma\in\P_ p} (-1)^{f(u_1)+f(v_1)} \textsf{E}_{\varsigma(u_1)} \textsf{E}_{u_2} \textsf{E}_{v_2}^* \textsf{E}_{\varsigma(v_1)}^* \\
&=\frac{\delta_{u_1,v_1}}{({ m}^{ p}-1)!}\sum_{\varsigma\in\P_ p} \textsf{E}_{\varsigma(u_1)} \textsf{E}_{u_2} \textsf{E}_{v_2}^* \textsf{E}_{\varsigma(u_1)}^* \\
& =\delta_{u_1,v_1} \gamma_ p( \textsf{E}_{u_2} \textsf{E}_{v_2}^*)\\
& =\delta_{u_1,v_1} \gamma_ p( \textsf{E}_{u_2})\gamma_p(\textsf{E}_{v_2})^*\quad (\text{as }\gamma_p \text{ is a C*-homomorphism}).
\end{align*}
We are done. 
\end{proof}

The following lemma is really of higher rank flavour, and it is crucial in the proof of Lemma \ref{L:lem} below. 

\begin{lem}
\label{L:id}
Let $u,v\in\Fth$ be such that $d(u)\wedge d(v)= 0$ and $p=d(u)+d(v)$.  Suppose that $\textsf{E}$ is a Cuntz-type representation of $\Fth$.
Then we have the following. 
\begin{itemize}
\item[(i)]
$\displaystyle
\alpha_ p(\textsf{E}_u\textsf{E}_v^*)=\frac{1}{ m^ p}\sum_{(v',u')\in\Lambda^{\text{min}}(u,v)} \gamma_ p (\textsf{E}_{v'}^*\textsf{E}_{u'}).
$
\\
Here we use the convention: if $|\Lambda^{\text{min}}(u,v)|=0$, then $\textsf{E}_{v'}^*\textsf{E}_{u'}:=0$.

\item[(ii)]
If $\Fth$ has the little pull-back property, then 
\[
\alpha_ p(\textsf{E}_u\textsf{E}_v^*)=0 \text{ or }\frac{1}{ m^ p}\gamma_ p (\textsf{E}_{u''}\textsf{E}_{v''}^*),
\]
for a unique pair $(u'', v'')\in\Fth\times \Fth$ with $d(u'')=d(u)$ and $d(v'')=d(v)$. 
\end{itemize}
\end{lem}

\begin{proof}
Suppose that $\textsf{E}$ is a Cuntz-type representation of $\Fth$. 

(i) 
Similar to Lemma \ref{L:2basic}, we prove it from the following calculations.
\begin{align*}
& \alpha_ p(\textsf{E}_u\textsf{E}_v^*)\\
&=\frac{1}{2^{ m^ p}( m^ p)!}\sum_{f\in\Omega_ p,\, \varsigma\in\P_ p}\left(\sum_{ \mu\in\Lambda^ p} (-1)^{f( \mu)} \textsf{E}_{\varsigma( \mu)} \textsf{E}_ \mu^*\cdot \textsf{E}_u\textsf{E}_v^*\cdot 
\sum_{ \nu\in\Lambda^ p} (-1)^{f( \nu)} \textsf{E}_{ \nu} \textsf{E}_{\varsigma( \nu)}^*  \right)\\
&=\frac{1}{( m^ p)!}\sum_{\varsigma}\left(\sum_{ \mu\in\Lambda^ p} \textsf{E}_{\varsigma( \mu)} \textsf{E}_ \mu^*\cdot \textsf{E}_u\textsf{E}_v^*\cdot \textsf{E}_{ \mu} 
\textsf{E}_{\varsigma( \mu)}^*  \right)
\end{align*}
We now write $\mu$ as $ \mu_1 \mu_2$  with $d( \mu_1)=d(u)$ and $d( \mu_2)=d(v)$. Continuing above, one has 
\begin{align*}
& \alpha_ p(\textsf{E}_u\textsf{E}_v^*)\\
&=\frac{1}{( m^ p)!}\sum_{\varsigma}\left(\sum_{ \mu\in\Lambda^ p} \textsf{E}_{\varsigma( \mu)} \textsf{E}_{ \mu_1 \mu_2}^*\cdot \textsf{E}_u\textsf{E}_v^*\cdot \textsf{E}_{ \mu_1 \mu_2} 
\textsf{E}_{\varsigma( \mu)}^*  \right)\\
&=\frac{1}{( m^ p)!}\sum_{\varsigma}\left(\sum_{ \mu_2\in\Lambda^{d(v)}} \textsf{E}_{\varsigma(u\mu_2)} \textsf{E}_{\mu_2}^*\cdot \textsf{E}_v^* \textsf{E}_{ u} \cdot \textsf{E}_{ \mu_2} 
\textsf{E}_{\varsigma( u\mu_2)}^*  \right)\\
&=\frac{1}{( m^ p)!}\sum_{\varsigma}\left(\sum_{ \mu_2\in\Lambda^{d(v)}} \sum_{(v',u')\in\Lambda^{\text{min}}(u,v)}\textsf{E}_{\varsigma(u\mu_2)} \textsf{E}_{\mu_2}^*
\cdot \textsf{E}_{u'} \textsf{E}_{v'}^* \cdot \textsf{E}_{ \mu_2} \textsf{E}_{\varsigma( u\mu_2)}^*  \right)\\
&=\frac{1}{( m^ p)!}\sum_{\varsigma}\left( \sum_{(v',u')\in\Lambda^{\text{min}}(u,v)}\textsf{E}_{\varsigma(uv')}\cdot \textsf{E}_{v'}^*\textsf{E}_{u'} \cdot
 \textsf{E}_{\varsigma( uv')}^*  \right)\\
&=\frac{1}{( m^ p)!}\sum_{(v',u')\in\Lambda^{\text{min}}(u,v)}\left(\sum_{\varsigma}\textsf{E}_{\varsigma(uv')}\cdot \textsf{E}_{v'}^*\textsf{E}_{u'} \cdot
\textsf{E}_{\varsigma( uv')}^*  \right)\\
&=\frac{1}{m^ p}\sum_{(v',u')\in\Lambda^{\text{min}}(u,v)}\gamma_p(\textsf{E}_{v'}^*\textsf{E}_{u'}).
\end{align*}

(ii) Suppose that $\Fth$ has the little pull-back property. Then we have a most one pair $(v',u')$ in $\Lambda^{\text{min}}(u,v)$, and
a most one pair $(u'',v'')$ in $\Lambda^{\text{min}}(v',u')$.
Thus $ \alpha_ p(\textsf{E}_u\textsf{E}_v^*)=\gamma_ p (\textsf{E}_{u''}\textsf{E}_{v''}^*)$ if  
$|\Lambda^{\text{min}}(u,v)|=|\Lambda^{\text{min}}(v',u')|=1$ and $(u'',v'')\in \Lambda^{\text{min}}(v',u')$, 
and $ \alpha_ p(\textsf{E}_u\textsf{E}_v^*)=0$, otherwise. 
\end{proof}

For $n\in\bN^\textsf{k}$, recall that $\X_n$ is the algebraic part of the spectral subspace $\ran\Phi_n$ of  the gauge action of $\O_\theta$. 
It follows from the definition of $\alpha_ p$ directly that $\X_ n$ is invariant for $\alpha_ p$. We record it below for later reference. 
 
\begin{lem}
\label{L:predeg}
For $ p\in\bN^{\textsf{k}}$ and $ n\in \bZ^\textsf{k}$, we have 
$
\alpha_ p(\X_ n)\subseteq \X_ n.
$
\end{lem}

In the rest of this section, let us fix the endomorphism $\gamma_p$ and the averaging operator $\alpha_p$ on $\O_\theta$ as follows: 
\begin{align*}
 \gamma_ p(A)&=\sum_{u\in\Lambda^p}  s_u A  s_u^*\qforal A\in\O_\theta,\\
\alpha_{ p}(A)
&=\frac{1}{2^{{ m}^{ p}}({ m}^{ p})!} \sum_{f\in\Omega_{{ p}}, \, \varsigma\in\P_{{ p}}} U_{(f,\varsigma)} A  U_{(f,\varsigma)}^*\qforal A\in\O_\theta,
\end{align*}
where $U_{f,\varsigma}=\sum_{u\in\Lambda^{ p}}(-1)^{f(u)}  s_{\varsigma(u)} s_u^*$.

\begin{lem} 
\label{L:lem}
Suppose that $\Fth$ has the little pull-back property, and let $\epsilon>0$.
\begin{itemize}

\item[(i)] If $A\in \X_ n$ with $ n\ne 0$, then there exists 
$\alpha\in\textsf{Ave}(\O_\theta, \U(\X_ 0))$ such that $\|\alpha(A)\|<\epsilon$. 

\item[(ii)] If $A=\sum_{ n \ne  0} A_ n$ in ${}^\circ{\O_\theta}$ with $A_ n\in\X_ n$, then there is 
$\alpha\in\textsf{Ave}(\O_\theta, \U(\X_ 0))$ such that $\|\alpha(A)\|<\epsilon$. 
\end{itemize}
\end{lem}

\begin{proof}  
Let $\epsilon>0$ be given.

(i)  Let $A\in \X_ n$ with $ n\ne  0$. 
We divide the proof into three steps. 

\textit{Step 1.}
Clearly, $A$ is spanned by the standard generators $ s_u s_{w_1} s_{w_2}^* s_v^*$, where $u,v,w_1,w_2\in\Fth$, 
$d(u)\wedge d(v)= 0$,  and $d(w_1)=d(w_2)=:q$, say.
Thanks to the defect free property of $s$, we can (and do) always assume that all summands of $A$ have a uniform such common degree $q$. 
(Actually, one can take it as the join $\vee  q$ of $ q$'s in the summands.) 

Notice that for a fixed $ n\in\bZ^\textsf{k}$, the equation
$
 k- l= n\text{ with }  k,  l\in \bN^\textsf{k} \text{ and }  k\wedge  l= 0
$
has a unique solution $ k= n_+$ and $ l= n_-$. 
Thus one can write $A$ as follows:
\[
A=\sum_{u\in\Lambda^{ n_+},\, v\in\Lambda^{ n_-},\, w_1, w_2\in\Lambda^ q} a_{u,v,w_1,w_2} s_u s_{w_1} s_{w_2}^* s_v^*.
\]
For convenience, set
\[
\S:=\spn\{ \gamma_{q}( s_u) \gamma_{q}( s_v)^*: u\in\Lambda^{ n_+},\, v\in\Lambda^{ n_-}\}.
\] 
From Lemma \ref{L:2basic} it follows that 
$\alpha_{q}(A)\in \fA$. Let us assume that 
\[
\alpha_{q}(A)=\sum_{u,v} a_{u,v}\gamma_{q}( s_u) \gamma_{q}( s_v)^*,
\]
where of course there are only finitely many coefficients $a_{u,v}$, which are not equal to $0$. 

\smallskip
\textit{Step 2.} 
Let $ p= n_++ n_-$. 
Applying Lemma \ref{L:id} (ii) to the Cuntz-type representation $\textsf{E}^1:= \gamma_ q(s)$ and using what we obtain in Step 1, we conclude that 
\[
\alpha_{p}^{\textsf{E}^1}(\alpha_ q(A))=\frac{1}{ m^ p} \gamma_{p}^{\textsf{E}^1}(A_1),
\]
where 
$A_1\in\S$
and its coefficients either come from those of $\alpha_q(A)$ or vanish.  
Straightforward calculations show that 
\[
 \gamma^{\textsf{E}^1}_ p\circ  \gamma_{ q}= \gamma_{ p+ q}.
\]

Since $ \gamma_{ p+ q}(s)=:\textsf{E}^2$ is a Cuntz-type representation of $\Fth$, applying Lemma \ref{L:id} (ii) again to $\textsf{E}^2$ yields 
\[
\alpha_ p^{\textsf{E}^2}(\alpha_{p}^{\textsf{E}^1}(\alpha_ q(A)))
=\frac{1}{ m^{ p}} \alpha^{\textsf{E}^2}_ p(\gamma_{p}^{\textsf{E}^1}(A_1))
=\frac{1}{ m^{2p}} \gamma^{\textsf{E}^2}_ p(\gamma_{p}^{\textsf{E}^1}(A_2))
=\frac{1}{ m^{2 p}} \gamma_{2 p}(A_2),
\]
where $A_2$ shares the same properties of $A_1$.

\textit{Step 3.}
Continuing the above process in an obvious way, we obtain Cuntz-type representations $\textsf{E}^1, \ldots, \textsf{E}^n$ 
and averaging mappings $\alpha_ p^{\textsf{E}^1},\ldots, \alpha_ p^{\textsf{E}^n}$ in $\textsf{Ave}(\O_\theta, \U(\X_ 0))$ such that 
\[
 \alpha_ p^{\textsf{E}^n}\circ\cdots\circ \alpha_ p^{\textsf{E}^1}\circ \alpha_ q(A)=\frac{1}{ m^{n p}} \gamma_{n p}(A_n),
\]
where $A_n\in\S$,
and its coefficients either come from those of $\alpha_q(A)$ or vanish.  
Let $\alpha= \alpha_ p^{\textsf{E}^n}\circ\cdots\circ \alpha_ p^{\textsf{E}^1}\circ \alpha_ q$,
which is clearly in $\textsf{Ave}(\O_\theta, \U(\X_ 0))$. 
By the definition of $\alpha_q(A)$, one has $\sum_{u,v}|a_{u,v}|$ is finite.  
It follows from the properties of $A_n$ that 
$\|A_n\| \le \sum|a_{u,v}|<\infty$ as $\|\gamma_{m}( s_{u} s_{v})^*\|\le 1$ for all $m\in\bN^{\textsf{k}}$. 
Hence, when $n$ is large enough, we have 
\[
\|\alpha(A)\|=\frac{1}{ m^{n p}}\| \gamma_{n p}(A_n)\|\le \frac{1}{ m^{n p}}\|A_n\|\le \frac{1}{ m^{n p}}\sum|a_{u,v}|<\epsilon.
\]
Here we use the facts that $ \gamma_{np}$ is isometric since it is an endomorphism on $\O_\theta$ and $\O_\theta$ is simple.

\medskip

(ii) This is completely similar to the proof of \cite[Proposition 4(iii)]{Arc80}. Let $A=\sum_{ n \ne  0} A_ n$ in ${}^\circ{\O_\theta}$ with $A_ n\in\X_ n$. 
We label the set of non-zero summands $\{A_ n\ne 0:  n\ne  0 \}$ of $A$ as $\{A_1, \ldots, A_L\}$. 
So $L$ is finite as $A\in{}^\circ\O_{\theta}$. 

By (i) above, there is $\alpha_1\in \textsf{Ave}(\O_\theta, \U(\X_ 0))$ such that $\|\alpha_1(A_1)\| <\frac{\epsilon}{L}$. 
Apply (i) above to $\alpha_1(A_2)$ by using Lemma \ref{L:predeg},  
and we get $\alpha_2 \in\textsf{Ave}(\O_\theta, \U(\X_ 0))$ such that $\|\alpha_2\circ\alpha_1(A_2)\| <\frac{\epsilon}{L}$.
Continuing this way, we obtain averaging operators $\alpha_1, \ldots, \alpha_{L}$ in $\textsf{Ave}(\O_\theta, \U(\X_ 0))$ such that  
\[
\|\alpha_{L}\circ\cdots\circ\alpha_1(A_{L})\| <\frac{\epsilon}{L}.
\]
Hence 
\[
\|\alpha_{L}\circ\cdots\circ\alpha_1(A_i)\| <\frac{\epsilon}{L}\quad (1\le i\le L),
\]
as every $\alpha_i$ is contractive. Therefore, $\alpha=\alpha_{L}\circ\cdots\circ\alpha_1$ belongs to $\textsf{Ave}(\O_\theta, \U(\X_ 0))$ and 
satisfies 
\[
\|\alpha(A)\|<\epsilon,
\]
as required. 
\end{proof}

\begin{rem}
Lemma \ref{L:lem} does not hold true in general if $\Fth$ does not satisfy the little pull-back property. 
For instance, consider a periodic $\textsf{k}$-graph $\Fth$. It is shown in \cite{DYrepk} that 
there is a central unitary $W=\sum_{u\in\Lambda^a} s_us_{\gamma(u)}^*$ in $\O_\theta$, where $\gamma$ is a bijection from $\Lambda^a$
to $\Lambda^b$ with $a,b>0$ in $\bN^\textsf{k}$ and $a\wedge b=0$. 
It is rather easy to check that 
$\alpha(W)=W$ for \textit{every} averaging operator $\alpha$. So one always has $\|\alpha(W)\|=1$.
\end{rem}

\subsection{The Dixmier property of $\O_\theta^\sigma$}

We already know that $\O_\theta$ has the Dixmier property since it is simple and has no tracial state (cf. Section \ref{S:pre}). 
The following shows that $\O_\theta$ has a stronger property when $\Fth$ has the little pull-back property. 

\begin{prop}
\label{P:SDixmier}
Suppose that $\Fth$ has the little pull-back property. 
Then for each $A\in\O_\theta$, 
\[
\ol{\{\alpha(A): \alpha\in \textsf{Ave}(\O_\theta, \U(\X_ 0))\}}^{\|\cdot\|}\cap \bC I \ne \mt.
\]
\end{prop}

\begin{proof}
The proof here is borrowed from \cite[Theorem 5]{Arc80}. Let $A\in{}^\circ\O_{\theta}$ and arbitrarily choose $\epsilon>0$. It is known that 
$A=\sum_{ n}A_ n$ with $A_ n \in\X_ n$ (cf., e.g., \cite{DYrepk, Raeburn}). 
By Lemma \ref{L:lem}, there is $\alpha_1\in\textsf{Ave}(\O_\theta, \U(\X_ 0))$ such that
$\|\alpha_1(A-A_ 0)\|<\epsilon/2$. Also, there is $p\in\bN$ such that $\alpha_1(A_ 0)\in\fF_p$, which is isomorphic to 
$M_{\Pi_{i=1}^\textsf{k} m_i^p}$. But $M_{\Pi_{i=1}^\textsf{k} m_i^p}$ has the Dixmier property (cf., e.g., \cite{Arc80}). So there are $\lambda\in\bC$ and  
$\alpha_2\in\textsf{Ave}(\O_\theta, \U(\X_ 0))$ such that 
$\|\alpha_2\circ\alpha_1(A_ 0-\lambda)\|<\epsilon/2$. Thus $\|\alpha_2\circ\alpha_1(A)-\lambda\|<\epsilon$. 
Then apply \cite[Lemma 2.8]{Arc77} to finish our proof.
\end{proof}

\begin{cor}
\label{C:DixPro}
Suppose that $\Fth$ has the little pull-back property. 
Let $\A$ be a unital C*-subalgebra of $\O_\theta$ such that $\X_ 0\subset \A$. Then 
$\A$ is simple and has the Dixmier property. 
\end{cor}

\begin{proof}
Since $\X_ 0\subset \A$, clearly $\U(\X_0)\subset \U(\A)$.
Hence the restriction $\alpha_i|_\A$ is an averaging mapping on $\A$ for every $\alpha_i$'s in the proof of Proposition \ref{P:SDixmier}.
It follows from Proposition \ref{P:SDixmier} that $\A$ has the weak Dixmier property. Therefore, $\A$ is simple and has the 
Dixmier property by \cite{HaaZsi84, Rie82} (also cf. Section \ref{SS:DP}), 
\end{proof}


\begin{cor}
\label{C:unitr}
Suppose that either $\Fth$ has the pull-back property, or that the intrinsic group of $\Fth$ has rank $0$.
Then $\O_\theta^\sigma$ has a unique tracial state. 
\end{cor}

\begin{proof}
If the intrinsic group of $\Fth$ has rank $0$, clearly the conclusion holds as $\O_\theta^\sigma=\fF$ by 
Proposition \ref{P:Ofixed}. 

Suppose $\Fth$ has the little pull-back property. On one hand, it follows from Proposition \ref{P:Ofixed} and Corollary \ref{C:DixPro} 
that $\O_\theta^\sigma$ has at most one tracial state. On the other hand, 
it is not hard to check that it has a tracial state, which is extended from 
the tracial state $\tau$ of $\fF$. 
\end{proof}

\section{The factoriality and type of $\pi(\O_\theta)''$}
\label{S:char}

We first recall the definition of KMS-states for C*-dynamical systems from \cite[Chapter 5]{BraRob}.
Let $(\fA, \bR, \varrho)$ be a C*-dynamical system. The state $\varphi$ over $\fA$ is said to be a
\textit{$\varrho$-KMS state  at value $\beta\in\bR$}, or a \textit{$(\varrho, \beta)$-KMS state}, if
\begin{align}
\label{E:KMS}
\varphi(AB)=\varphi(B\varrho_{i\beta}(A))
\end{align}
for all $A,B$ entire for $\rho$.
A $\varrho$-KMS state at value $\beta=-1$ is simply called a \textit{$\varrho$-KMS state}.

Recently, Huef, Laca,  Raeburn and Sims \cite{HLRS13} study in detail the KMS states of some dynamical systems associated to arbitrary finite \textsf{k}-graphs without sources.
(Also refer to it for a nice introduction to KMS states.) In
particular, it is shown that  if $\{\ln m_\ell: 1\le \ell\le \textsf{k}\}$ is rationally independent,
then $\O_\theta$ has a unique 
$\sigma$-KMS state (\cite[Theorem 7.2]{HLRS13}, also see \cite{Yang1} for $\textsf{k}=2$). 
Moreover, they remark in \cite[Example 7.3]{HLRS13} that the 
condition of $\{\ln m_\ell: 1\le \ell\le \textsf{k}\}$ being rationally independent is necessary. 
But one should notice that the example provided there is a \textit{periodic} single vertex 2-graph. Actually, it is the flip algebra using the terminology of \cite{DPY1, DPYdiln, DYperiod}. 

The following result generalizes \cite[Theorem 7.2]{HLRS13} in the case of single vertex \textsf{k}-graphs and
\cite[Proposition 4.4]{Yang2}.  
The special case of $\theta=\id$ is proven in \cite{Yang2} using a more direct method. 

\begin{prop}
\label{P:unitra}
Suppose that either $\Fth$ has the little pull-back property, or the intrinsic group of $\Fth$ has rank $0$. 
Then the state $\omega$ is the unique $\sigma$-KMS state over $\O_\theta$. 
\end{prop}

\begin{proof}
As in \cite[Proposition 5.4]{Yang1}, it is easy to see that $\omega$ is a $\sigma$-KMS over $\O_\theta$. It is left to show that $\omega$ is the only one. 

If the intrinsic group of $\Fth$ has rank $0$, namely, $\{\ln m_\ell: 1\le \ell\le \textsf{k}\}$ is rationally independent, this case is done by \cite[Theorem 7.2]{HLRS13}.  

We now suppose that $\Fth$ has the little pull-back property.
Consider the set $\{\ln m_\ell: 1\le \ell\le \textsf{k}\}$. Assume that it  has exactly $n$ rationally independent elements. Clearly $1\le n\le \textsf{k}$ as all $m_\ell$'s are $>1$. 
WLOG, say $\ln m_1, \ldots, \ln m_n$ are rationally independent. 
Therefore, one can simply rewrite the action $\sigma$ as follows:
\begin{align}
\label{E:sigma}
\sigma_t=\gamma_{(m_1^{-it},\ldots, m_\textsf{k}^{-it})}=\gamma_{(z_1, ..., z_n,\, f_1, \ldots,\, f_{\textsf{k}-n})},
\end{align}
where $z_\ell=m_\ell^{-it}$ ($1\le \ell\le n$), and $f_j=f_j(z_1,\ldots, z_n)$ ($1\le j\le \textsf{k}-n)$ is a $\bT$-valued function of $z_1,...,z_n$
(with rational coefficients).  
Since $\{\ln m_\ell:1\le \ell\le  n\}$ is rationally independent, $\sigma$ yields an action of $\bT^n$ on $\O_\theta$. 
So there is a faithful conditional expectation $\E$ from $\O_\theta$ onto $\O_\theta^\sigma$ (cf. \cite{BO08}).
From Section 2, it is not hard to see that $\E$ is given by 
\begin{align}
\label{E:sigam}
\E(A)=\int_{\bT^n} \gamma_{(z_1, ..., z_n,\, f_1, \ldots,\, f_{\textsf{k}-n})}(A)\, dz\qforal A\in \O_\theta. 
\end{align}

Now we assume that $\phi$ is also a $\sigma$-KMS over $\O_\theta$. 
On one hand, by Corollary \ref{C:unitr} we have
\begin{align}
\label{E:phiome}
\phi\circ \E= \tau\circ \Psi\circ \E=\omega\circ \E,
\end{align}
where $\Psi$ is the canonical faithful expectation of $\O_\theta^\sigma$ onto $\fF$, and $\tau\circ \Psi$ is the unique tracial state of $\O_\theta^\sigma$. On the other hand, every
$\sigma$-KMS sate is invariant with respect to $\sigma$ (\cite{BraRob}). So 
$\omega=\omega\circ \sigma$ and 
$\phi=\phi\circ \sigma$. 
It now follows from \eqref{E:sigma} and \eqref{E:sigam} that $\phi=\phi\circ \E$ and $\omega=\omega\circ\E$. This
implies $\phi=\omega$ from \eqref{E:phiome}.
\end{proof}

\begin{rem}
As in \cite[Lemma 4.1]{Yang2}, every KMS state over $\O_\theta$ for $\sigma$ is a $\sigma$-KMS state. 
One can also borrow the proof from \cite[Proposition 5.4 (iii)]{Yang1} in the case of rational independence. 
\end{rem}

Before giving our main result of this paper, we need to introduce more notation.
Let $\fM$ be a $\sigma$-finite von Neumann algebra. The \textit{Connes spectrum $\S(\fM)$} is the intersection over all faithful
normal states of the spectra of their corresponding modular operators \cite{BraRob}.
A type III factor $\fM$ is said to be of
\begin{alignat*}{3}
type\ III_0 \quad \text{if} & \quad \S(\fM)=\{0, 1\};\\
type\ III_1 \quad \text{if} &\quad  \S(\fM)=[0, \infty);\\
type\ III_\lambda \quad \text{if} &\quad  \S(\fM)=\{0, \lambda^n:n\in \bZ\} \quad (0<\lambda<1).
\end{alignat*}
See, for example, \cite{BraRob, KadRin, Take} for more information.

\begin{thm}
\label{T:char}
Let $G$ be the intrinsic group of $\Fth$.
\begin{itemize}
\item[(i)] If $\rank G=0$, then $\pi(\O_\theta)''$ is an AFD type III$_1$ factor. 

\item[(ii)] If $\Fth$ has the little pull-back property, 
then $\pi(\O_\theta)'' $ is an AFD type III factor.
Moreover, the type of  $\pi(\O_\theta)''$ is determined as follows: 
\begin{itemize}
\item[(III$_1$)] If $\rank G\ne \textsf{k}-1$, 
then $\pi(\O_\theta)''$ is of type III$_1$. 

\item[(III$_\lambda$)] If $\rank G =\textsf{k}-1$,
then $\pi(\O_\theta)''$ is of type III$_\lambda$, where $0<\lambda<1$ is determined by $m_1,\ldots, m_{\textsf{k}}$.
\end{itemize}
\end{itemize}
\end{thm}


\begin{proof}
By Proposition \ref{P:unitra} and \cite[Theorem 5.3.30]{BraRob}, $\pi(\O_\theta)''$ is factor.
Clearly, it is of type III. Since $\O_\theta$ is known to be amenable, $\pi(\O_\theta)''$ has to be AFD.
Hence $\pi(\O_\theta)''$ is an AFD factor. 

We now determine the type of $\pi(\O_\theta)''$. The following proof is similar to \cite[Theorem 5.2]{Yang2}. 
By Corollary \ref{C:unitr}, $\pi(\O_\theta^{\sigma})''$ is a factor (cf., e.g., \cite{Was91}). It follows
from \cite[Section 28.3]{Str} that the Connes spectrum $\S(\pi(\O_\theta)'')$ coincides
with the spectrum $\rm{Sp}(\Delta) $ of the modular operator $\Delta$ of $\omega$. That is, $\S(\pi(\O_\theta)'')=\rm{Sp}(\Delta)$.

\smallskip
(III$_1$) Suppose that $\rank G\ne \textsf{k}-1$. Then
$\frac{\ln m_i}{\ln m_j}\not\in\bQ$ for some $i,j$. So
\begin{align*}
\rm{Sp}(\Delta)
&=\cl {\{ m^ n: n\in\bZ^\textsf{k}\}} \\
&\supseteq \cl {\{m_i^{n_i}m_j^{n_j}:n_i, n_j\in\bZ\}}\\
&=[0,\infty),
\end{align*}
where, for a set $A\subseteq\bR$, $\cl A$ means the closure of $A$. Hence $\pi(\O_\theta)''$ is of type III$_1$. 

(III$_\lambda$) Suppose that  $\rank G=\textsf{k}-1$. 
Equivalently, one has $\frac{\ln m_1}{\ln m_j}\in\bQ$ for $2\le j\le \textsf{k}$. 
So there are natural numbers $a_j$ and $b_j$ ($2\le j \le \mathsf{k}$) such that 
$m_1^{a_j}=m_j^{b_j}$ with $\gcd(a_i,b_j)=1$  for $2\le j\le \textsf{k}$. 
Then 
\begin{align*}
\rm{Sp}(\Delta)
&=\cl {\{ m^ n: n\in\bZ^\textsf{k}\}} \\
&=\cl \left\{m_1^{n_1}\Pi_{j=2}^\textsf{k} m_1^{(a_j/b_j)n_j}:n_1,...,n_\textsf{k}\in\bZ\right\} \\
&=\cl \left\{m_1^{n_1+\sum_{j=2}^\textsf{k} (a_jn_j/b_j)}:n_1,...,n_\textsf{k}\in\bZ\right\} \\
&=\cl\left\{\left(m_1^{-\frac{1}{b_2\cdots b_\textsf{k}}}\right)^{-b_2\cdots b_\textsf{k} n_1-\sum_{j=2}^\textsf{k} a_jn_j\Pi_{\ell \ne j}b_\ell}: n_1,...,n_\textsf{k}\in\bZ\right\} \\
&=\left\{0, \left(m_1^{-\frac{1}{b_2\cdots b_\textsf{k}}}\right)^N: N\in \bZ\right\}
\end{align*}
So $\pi(\O_\theta)''$ is of type III$_\lambda$ with $\lambda=m_1^{-\frac{1}{b_2\cdots b_\textsf{k}}}$. 
 \end{proof}

\subsection*{Acknowledgement} The author would like to thank the referee for his/her useful comments. 


\end{document}